\documentclass{amsart}
\usepackage{pifont,latexsym,ifthen,amsthm,calc,amsfonts,amssymb,amsbsy,amsmath}
\usepackage[square,comma]{natbib}
\newcommand{\abs}[1]{\left\lvert #1\right\rvert}

\newcommand{\norm}[1]{\left\lVert #1\right\rVert}
\newcommand{\ffield}{\mathbb{F}_q(\!(t^{-1})\!)}
\newtheorem{theorem}{Theorem}[section]
\newtheorem{lemma}[theorem]{Lemma}
\newtheorem{corollary}[theorem]{Corollary}

\begin{document}
\title{A sum-product theorem in function fields}

\author{Thomas F. Bloom and Timothy G. F. Jones}
\begin{abstract}
Let $A$ be a finite subset of $\ffield$, the field of Laurent series in $1/t$ over a finite field $\mathbb{F}_q$. We show that for any $\epsilon>0$ there exists a constant $C$ dependent only on $\epsilon$ and $q$ such that $\max\left\{|A+A|,|AA|\right\}\geq C |A|^{6/5-\epsilon}$. In particular such a result is obtained for the rational function field $\mathbb{F}_q(t)$. Identical results are also obtained for finite subsets of the $p$-adic field $\mathbb{Q}_p$ for any prime $p$.
\end{abstract}	
\address{Department of Mathematics\\University of Bristol\\
University Walk\\Clifton\\ Bristol BS8 1TW\\United Kingdom}
\email{matfb@bristol.ac.uk and tgf.jones@bristol.ac.uk}
\thanks{The authors are both supported by an EPSRC doctoral training grant.}
\maketitle

\section{Introduction}

Let $A$ be a non-empty finite subset of a ring. Consider the sumset 

\[A+A=\left\{a+b:a,b \in A\right\}\]
and the product set 
\[AA=\left\{ab:a,b \in A\right\}.\]

It is conjectured that, provided $A$ does not contain `too many' zero-divisors, at least one of $A+A$ or $AA$ must be large, in some sense to be discussed later, whenever $A$ is not close to being a subring. In this paper we make progress in this direction in the settings of global function fields and the $p$-adic numbers.

In what follows we say that $\delta$ is permissible for a collection of sets $\mathcal{A}$ if for all $A\in\mathcal{A}$ and $\epsilon>0$
\[\max(|AA|, |A+A|)\gg_{\epsilon} |A|^{1+\delta-\epsilon}.\]
\cite{ErSz:1983} conjectured that $1$ is permissible for all finite sets of integers. The best result towards this to date is due to \cite{So:2009}, who showed that $1/3$ is permissible for all finite sets of real numbers. \cite{KoRu:2012} have recently extended this result to the complex numbers. Previous results in the real and complex settings can be found in \cite{El:1997,Fo:1998,Na:1997,So:2005a,So:2005b}.

Much has also been achieved in the finite field setting. In a finite field of prime order, \cite{BoKaTa:2004} proved the existence of an absolute constant $\delta>0$ which is permissible for all subsets of $\mathbb{F}_p$ that are not too close to being the entire field $\mathbb{F}_p$, for all primes $p$. More generally, \cite{Ta:2009} extended this result to general rings, covering all finite subsets that are not too close to being a subring. 

Garaev made the Bourgain-Katz-Tao result explicit in two cases: when $|A|>p^{2/3}$ \citep{Ga:2008} and when $|A|<p^{1/2}$  \citep{Ga:2007}. The former estimate is sharp, but the latter has been subsequently improved \citep[see][]{KaSh:2008a,Sh:2008,BoGa:2009,Li:2011}. The most recent result is that $1/11$ is permissible for all sets $A\subset\mathbb{F}_p$ such that $|A|<p^{1/2}$, due to \cite{Ru:2012}. \cite{LiRo:2011} built on a technique of \cite{KaSh:2008b} to extend this estimate to any finite field, not necessarily of prime order, so long as $A$ is not too close to being a subfield.

In this paper we consider sum-product estimates for $\ffield$, the field of Laurent series over a finite field $\mathbb{F}_q$. We shall show that for subsets of such fields $1/5$ is permissible. Our methods also work for the $p$-adic setting so that $1/5$ is also permissible for any prime $p$ and finite subset of $\mathbb{Q}_p$. In particular, combined with the previous results over $\mathbb{R}$, this implies that $1/5$ is permissible for all finite subsets of any local field.

We indicate possible applications to constructions in theoretical computer science below, for which it is likely only an appeal to the theorem over $\mathbb{F}_q[t]$ is necessary, which follows immediately since $\mathbb{F}_q[t]\subset\ffield$.

We will prove the following theorem. 

\begin{theorem}\label{theorem:functionfieldsumproduct}
For any finite $A\subset\ffield$ and any $\epsilon>0$ we have 

\[|A+A|^3|AA|^2 \gg_{\epsilon} q^{-1}|A|^{6-\epsilon}.\]
\end{theorem}
A sum-product result for $\ffield$ follows immediately.
\begin{corollary}
For any finite $A\subset\ffield$ and any $\epsilon>0$ we have 

\[\max \left\{|A+A|,|AA|\right\}\gg_{\epsilon,q} |A|^{1+\frac{1}{5}-\epsilon}\]
\end{corollary}

The sum-product exponent of $1/5$ for $\ffield$ lies between the $1/11$ known for finite fields and the $1/3$ known for $\mathbb{R}$. It is natural to conjecture that the correct answer for $\ffield$ is $1$.

We remark that the dependence on $q$ cannot be removed, since $\ffield$ contains $\mathbb{F}_q$ as a subfield, which is closed under both addition and multiplication. This contrasts with the finite field setting for sum-products, where we think of $q$ as being large, as it is an upper bound for the cardinality of our sets. Clearly any finite field sum-product result for which the constants depended on $q$ would be meaningless, since we could write everything as $O_q(1)$.

Aside from their intrinsic interest these theorems may have applications to constructions in theoretical computer science. Sum-product results over finite fields have seen recent applications to extractors; see for example \cite{Bo:2005}. A key idea in these previous applications is the observation that a string of $n$ bits, that is, an element of $\mathbb{F}_2^n$, can be interpreted as an element of the field $\mathbb{F}_{2^n}$ so that the full power of the sum-product machinery can be brought to bear. An alternative is to interpret it as an element of $\mathbb{F}_2[t]$. Given that the sum-product results now available in $\mathbb{F}_2[t]$ are better than those in $\mathbb{F}_{2^n}$ (with an exponent of $1/5$ rather than $1/11$) we expect that constructions along a similar line to those in \cite{Bo:2005} will be quantitatively stronger over $\mathbb{F}_2[t]$ rather than $\mathbb{F}_{2^n}$.

For number theory and geometric applications it is worth pointing out that any global function field, that is, a field of transcendence degree 1 over a finite field, can be embedded into $\ffield$ for some $q$, where $q$ depends only on the field of constants and genus of the function field. In particular Theorem~\ref{theorem:functionfieldsumproduct} will also hold for any finite subset of a global function field.

Finally, we remark that it is crucial to the result that the field of constants for $\ffield$ is finite. Obtaining a sum-product result in, say, $\mathbb{Q}(t)$ requires a different approach, and has been done with a non-explicit exponent by \cite{CrHa:2010} (whose methods would also likely lead to a sum-product result over $\mathbb{F}_q(t)$ but almost certainly with a weaker exponent).

The proof is written in the language of $\ffield$ for concreteness, but all that is required of the field we are working in is that is has a non-archimedean norm and furthermore has a finite residue field. In particular, the methods in this paper also prove the following sum-product theorem for the $p$-adic numbers $\mathbb{Q}_p$, where nothing was previously known. 

\begin{theorem}
For any prime $p$, finite $A\subset\mathbb{Q}_p$, and any $\epsilon>0$ we have 

\[|A+A|^3|AA|^2 \gg_{\epsilon,p} |A|^{6-\epsilon}.\]
\end{theorem}

Finally, we remark that, if desired, the arguments that follow could be adapted to cover the case of partial sum and product sets, giving sum-product estimates in terms of the additive or multiplicative energy directly. For applications to exponential sums or theoretical computer science this approach via energy would be quantitatively stronger.

The approach used in this paper is based on a geometric argument used by \cite{So:2005b} in the complex numbers, coupled with some striking structural properties of non-archimedean geometry.

The next section provides more of the necessary background on function fields, and explains the structure of the proof and the rest of the paper. 
 
\section{Function fields}\label{section:ff}
Let $\ffield$ be the field of Laurent series over a finite field $\mathbb{F}_q$ with $q$ elements. Recall that elements of $\ffield$ have the shape
\[\sum_{i\leq N}a_it^i,\]
where $a_i\in \mathbb{F}_q$, for some $N\in\mathbb{Z}$. Throughout this paper $q$ will be reserved for the size of $\mathbb{F}_q$ and $p$ for the characteristic of $\mathbb{F}_q$, so that $q$ is a power of $p$.

There is a non-archimedean norm on $\ffield$ given by $\norm{x}=q^{\deg x}$ where $\deg$ is the familiar degree map; that is, if
\[x=\sum_{i\leq N}a_it^i\]
and $a_N\neq 0$ then we say that $\deg x=N$. We define $\norm{0}=0$. This has the crucial non-archimedean property
\[\norm{x+y}\leq\max(\norm{x},\norm{y})\]
and furthermore $\norm{x}=\norm{-x}$. Both of these properties will be used frequently without further mention in what follows. As a consequence of the non-archimedean property $\ffield$ has an unusually rigid geometry, which will be exploited when proving sum-product estimates. A particular concern will be the behaviour of balls, which are sets of the form
	\[B(x,r)=\left\{y \in \ffield : \norm{x-y}\leq r\right\}.\]
We will call $r\in\mathbb{R}$ the radius of the ball $B(x,r)$. The non-archimedean property implies the following result which is considered standard.

\begin{lemma}\label{theorem:nonarchball}
If $B_1$ and $B_2$ are balls in $\ffield$ then either they are disjoint, or $B_1 \subset B_2$, or $B_2 \subset B_1$. If in addition $B_1$ and $B_2$ have the same radius then either they are disjoint or $B_1=B_2$.
\end{lemma}

\begin{proof}
Let $B_1=B(x,r)$ and $B_2=B(y,s)$. If there exists $a \in B(x,r) \cap B(y,s)$ then 
	\[\norm{x-y}\leq \max\left\{\norm{a-x},\norm{a-y}\right\}\leq \max\left\{r,s\right\}.\]	
If $r \leq s$ then this implies $B(x,r)\subset B(y,s)$ since if $b \in B(x,r)$ then 
	\[\norm{y-b}\leq \max\left\{\norm{y-x},\norm{b-x}\right\}\leq \max\left\{r,s\right\}=s.\]
Conversely if $s \leq r$ then $B(y,s)\subset B(x,r)$. Hence if $r=s$ then $B(x,r)=B(y,s)$.
\end{proof}

The proof of Theorem~\ref{theorem:functionfieldsumproduct} builds upon an approach of \cite{So:2005b} for sum-products in $\mathbb{C}$. When adapting this method, the non-archimedean geometry of function fields turns out to be a mixed blessing. 

First, the bad news. Solymosi's argument fails at a critical point in the function field setting, for the following reason. For each $a \in A$, let $a' \in A \setminus\left\{a\right\}$ be such that $\norm{a-a'}$ is minimal, and let $B_a$ be the ball of radius $\norm{a-a'}$ centred on $a$. Solymosi's method uses the crucial fact that a single complex number can be contained in at most $O(1)$ of the $B_a$. This fails spectacularly in $\ffield$, where an element could be contained in as many as $|A|$ of the $B_a$, as demonstrated by the following example: let 
	\[A=\left\{t^j:0 \leq j \leq n\right\}\]
so that
	\[B_{t^j}=\left\{x \in \ffield:|x|\leq q^j\right\}\]
for $j\geq 1$ and $B_{1}=B_{t}$, meaning that every one of the $|A|$ balls contains $0$ as an element. 

But all is not lost. In the example above, the astute reader will notice that $|A+A|\approx |A|^2$, and so a strong-sum product estimate holds despite the failure of Solymosi's argument. In fact we will be able to show that something like this is possible whenever Solymosi's argument fails, by considering a special type of structure to be defined in the following section: separable sets.

Separability is fairly unexciting in the complex setting, as it holds trivially for any non-empty finite set, but in the non-archimedean regime of function fields it is a stronger notion. The rigid geometry makes it harder to find separable sets, but where they do exist it will in fact imply the existence of large sumsets. The idea, therefore, is to show that a large separable set must exist whenever Solymosi's argument fails. Combining this with an analysis of separable sets as having large sumsets will lead to a proof of Theorem~\ref{theorem:functionfieldsumproduct}. 

In what follows, Section~\ref{section:separable} analyses separable sets and shows that their sumsets have maximal growth. Section~\ref{section:solymosi} then adapts Solymosi's proof from \cite{So:2005b} to establish that if $|A+A|$ and $|AA|$ are both small then there must exist a large separable set, and uses this to prove Theorem~\ref{theorem:functionfieldsumproduct}. 

\section{Separable sets}\label{section:separable}

A finite set $A \subset \ffield$ is {\em separable} if its elements can be indexed as
	\[A=\left\{a_1,\ldots,a_{|A|}\right\}\]
in such a way that for each $1 \leq j \leq |A|$ there is a ball $B_j$ with 
	\[A \cap B_j =\left\{a_1,\ldots,a_j\right\}.\] 
	
\begin{lemma}\label{theorem:analyseseparable}
If $A\subset\ffield$ is a finite separable set then 
	\[|k A|\gg_k|A|^{k}\]
for any natural number $k$.
\end{lemma}

\begin{proof}
Let $E_k(A)$ denote the $k$-fold additive energy of $A$, i.e. the number of solutions to
	\begin{equation}\label{eq:kenergy}
	a_1+\ldots+a_k=b_1+\ldots+b_k
	\end{equation}
with $a_i,b_i \in A$. For $x \in kA$ write $\mu(x)$ for the number of solutions to 
	\[x=a_1+\ldots+a_k.\]
By the Cauchy-Schwarz inequality
	\[|A|^{2k}= \left(\sum_{x \in kA}\mu(x)\right)^2\leq |kA|E_k(A)\] 
and so it suffices to show that $E_k(A)\ll_k |A|^k$. 

Say that a solution $(a_1,\ldots,b_k)$ to \eqref{eq:kenergy} is trivial if there are at most $k$ distinct elements of $A$ in $\{a_1,\ldots,b_k\}$. By elementary counting there are at most $O_k(|A|^k)$ trivial solutions, so it suffices to show that there are no non-trivial solutions.

Suppose for a contradiction that a non-trivial solution to \eqref{eq:kenergy} exists. Gathering terms gives an expression of the form 
	\begin{equation}\label{eq:energygathered}
	n_1c_1+\cdots +n_tc_t =0
	\end{equation}    
where the $c_i $ are distinct elements of $A$ and $n_i\in \mathbb{Z}\backslash\{0\}$. The assumption of non-triviality implies that $t \geq 2$. Additionally, note that
	\begin{equation}\label{eq:coeffsgathered}
	n_1+\cdots +n_t = 0.
	\end{equation}    
Indeed, after gathering terms on the left the different multiplicities $n_i$ must sum to zero, since there are the same number of terms on the left of \eqref{eq:kenergy} as on the right.

An easy calculation shows that since the solution is non-trivial there exists $n_j$ such that $\abs{n_j}=1$. Without loss of generality, let this be $n_1$. Since $A$ is separable and the $c_i$ are in $A$ we may relabel them if necessary to assume the existence of a ball $B(x,r)$ such that $c_1 \notin B$ but $c_2,\ldots,c_t \in B$. By \eqref{eq:energygathered},
	\begin{align*}
	\norm{c_1-x}&=\norm{n_1c_1-n_1x}\\
	&=\norm{n_2c_2+\cdots +n_tc_t + n_1x}.
	\end{align*}
By \eqref{eq:coeffsgathered} and the non-archimedean property it follows that
	\begin{align*}
	\norm{c_1-x}&=\norm{n_2(c_2-x)+\cdots +n_t(c_t-x)}\\
	&\leq \max \left\{\norm{c_2-x},\ldots,\norm{c_t-x}\right\}\\
	&\leq r
	\end{align*}
and hence $c_1 \in B(x,r)$ which is a contradiction, and the proof is complete.
\end{proof}

\section{Finding separable sets}\label{section:solymosi}

The goal in this section is to show that if the sumset and product set of a set $A$ are both small then $A$ must contain a large separable set. For this we adapt the argument of \cite{So:2005b} for complex sum-products discussed in Section~\ref{section:ff}. Note that all of the analysis remains in the $\ffield$ setting; indeed some of the facts of non-archimedean geometry deployed here are manifestly false in $\mathbb{C}$.

A couple of new definitions are required. For a finite set $A \subset \ffield$ and an element $a \in A$, define 
	\begin{align*}
	r_A(a)&=\min_{a'\in A\backslash\{a\}}|a-a'|\\
	B_A(a)&=B(a,r_A(a)).\\
	\end{align*}
Additionally, for any $n\geq 1$ we say that $C=(c_1,\ldots,c_n)\in A^n$ is an {\em $A$-chain} of length $n$ if $c_i\neq c_j$ for $1\leq i<j\leq n$ and 
	\[B_A(c_1)\subset \cdots \subset B_A(c_n).\]
We will be rather cavalier about using $C$ for both a chain, an ordered tuple of elements from $A$, and the unordered set of such elements; which is intended will be clear from the context. The following argument, a strengthened form of that found in \cite{So:2005b}, finds a large chain in $A$ as long as the sumset and partial product set are both small. If this condition were to fail then a suitable sum-product result would follow immediately.

\begin{lemma}\label{theorem:chains}
Let $A\subset\ffield$ be any finite set. Then $A$ contains an $A$-chain of cardinality at least \[\frac{\abs{A}^5}{2^7|A+A|^2|AA|^2(\log_2|A|)^3}.\]
\end{lemma}

\begin{proof}
For each $a \in A$ write $N(a)$ for the maximal length $N$ of an $A$-chain $C=(c_1,\ldots,c_{N})$ for which $c_{N}=a$. Note for future reference that 
	$$N(a) \leq \left|B_A(a) \cap A\right|$$ 
since if $C$ is such an $A$-chain then $C \subseteq A$ by definition and for each $c \in C$ we have $c \in B_A(c)\subseteq B_A(a)$.

It suffices to find $a \in A$ such that
	$$N(a)\geq \frac{\abs{A}^5}{2^7|A+A|^2|AA|^2(\log_2|A|)^3}.$$
Begin with a dyadic pigeonholing. For each $0 \leq j \leq \log_2 |A|$ define $A_j$ to be the set of $a \in A$ for which $2^j \leq N(a) < 2^{j+1} $. The $A_j$ partition $A$ and so
	\[\sum_{j=0}^{\log_2 |A|}|A_j|=|A|.\]
Hence there exists some $j_0$ for which $|A_{j_0}|\geq |A|/\log_2|A|$. We shall show that 
	$$2^{j_0}\geq \frac{\abs{A}^5}{2^7|A+A|^2|AA|^2(\log_2|A|)^3}.$$ 

To this end, say that a pair $(a,c)\in A \times A$ is additively good if 
	$$|(A+A)\cap (B_A(a)+c)|\leq \frac{2^{j_0+3}|A+A|}{|A_{j_0}|}$$	
and that $(a,d)\in A \times A$ is multiplicatively good if	
	$$|(AA)\cap (d\cdot B_A(a))| \leq \frac{2^{j_0+3}|AA|}{|A_{j_0}|}.$$	
Say that a quadruple $(a,b,c,d) \in A^4$ is good if
	\begin{eqnarray}
	&a \in A_j. \label{eq:con1} \\
	&b \in B_A(a) \cap  A. \label{eq:con2} \\
	&(a,c) \text{ is additively good.} \label{eq:con3} \\
	&(a,d) \text{ is multiplicatively good.} \label{eq:con4}
	\end{eqnarray}
	
Write $Q$ for the number of good quadruples. We shall bound $Q$ from below to obtain
	\begin{equation}\label{eq:tom1}
	Q \geq 2^{j_0-1} |A_{j_0}||A|^2
	\end{equation}
and bound it from above to obtain
	\begin{equation}\label{eq:tom2}
	Q \leq \frac{2^{2j_0+6}|A+A|^2|AA|^2}{|A_{j_0}|^2}. 
	\end{equation}
Comparing (\ref{eq:tom1}) and (\ref{eq:tom2}) will then give the required bound on $2^{j_0}$ since $|A_{j_0}|\geq |A|/\log_2|A|$. Let's first establish (\ref{eq:tom1}). For fixed $c \in A$ we have
	\begin{align*}
	\sum_{a \in A_{j_0}}|(A+A)\cap (B_A(a)+c)|&=\sum_{a \in A_{j_0}}\sum_{u \in A+A}\mathbf{1}_{u \in B_A(a)+c}\\
	&= \sum_{v \in A+A-c} \sum_{a \in A_{j_0}}1_{v \in B_A(a)}\\
	&= \sum_{v \in A+A-c}|C_{j_0}(v)|
	\end{align*}
where $C_{j_0}(v)$ is the set of $a \in A_{j_0}$ with $v \in B_A(a)$. 

Observe that the elements of $C_{j_0}(v)$ may be ordered to form an $A$-chain. This follows from Lemma \ref{theorem:nonarchball} since for any two $a,b \in C_{j_0}(v)$ we have $v \in B_A(a) \cap B_A(b)$ and so either $B_A(a) \subseteq B_A(b)$ or $B_A(b) \subseteq B_A(a)$. 

Now since $C_{j_0}(v) \subseteq A_{j_0}$ and $C_{j_0}(v)$ is an $A$-chain, there is an $a \in A_{j_0}$ for which 
	$$\left|C_{j_0}(v)\right| \leq N(a) \leq 2^{{j_0}+1}.$$ 
We therefore have
	$$\sum_{a \in A_{j_0}}|(A+A)\cap (B_A(a)+c)|\leq 2^{{j_0}+1}|A+A|$$
and hence 
	$$|(A+A)\cap (B_A(a)+c)|\leq \frac{2^{{j_0}+3}|A+A|}{|A_{j_0}|}$$
holds for at least $3|A_{j_0}|/4$ elements $a \in A_{j_0}$. So for fixed $c \in A$ there are at least $3|A_{j_0}|/4$ elements $a \in A_{j_0}$ for which $(a,c)$ is additively good.

By the same argument we may show that for fixed $d \in A$ there are at least $3|A_{j_0}|/4$ elements $a \in A_{j_0}$ for which $(a,d)$ is multiplicatively good.

Thus for any $c\in A$ and $d\in A$ there are at least $|A_{j_0}|/2$ elements $a \in A_{j_0}$ for which $(a,c)$ is additively good and $(a,d)$ is multiplicatively good, i.e. for which conditions \eqref{eq:con3} and \eqref{eq:con4} hold. Furthermore for each such $a \in A_{j_0}$ there are at least $2^{j_0}$ elements $b \in A$ for which condition \eqref{eq:con2} holds, since 
	$$2^{j_0} \leq N(a) \leq \left|B_A(a)\cap A\right|.$$ 
In total therefore,
	$$Q \geq |A|^2|A_{j_0}|2^{j_0-1}$$
which concludes the proof of \eqref{eq:tom1}.

We now prove \eqref{eq:tom2}. Note that the map
 $$(a,b,c,d)\mapsto (a+c,b+c,ad,bd)$$
is injective and so it suffices to bound the number of possibilities for this latter expression, subject to the constraint that $(a,b,c,d)$ is good. There are certainly at most $|A+A|$ possibilities for $a+c$ and at most $|AA|$ for $ad$, so it suffices to show that if these are fixed then there are at most $2^{j_0+3}|A+A|/|A_{j_0}|$ possibilities for $b+c$ and at most $2^{j_0+3}|AA|/|A_{j_0}|$ for $bd$.

First establish the bound on the number of $b+c$. Note that if 
	$$a+c=a'+c'$$
then either 
	$$B_A(a)+c \subseteq B_A(a')+c'$$ 
or 
	$$B_A(a')+c' \subseteq B_A(a)+c$$ 
since both sets are balls with the same centre $a+c$. 

As a consequence, if $G\subseteq A \times A $ is the set of additively good pairs $(a,c)$, then for any $x \in A\overset{G}+A$ there is a fixed additively good pair $(a_x,c_x)$ such that 
	$$B_A(a)+c \subseteq B_A(a_x)+c_x$$
whenever $a+c=x$ and $(a,c)$ is additively good. Thus if $a+c=x$ is the fixed first co-ordinate and $b+c$ is a possible second co-ordinate then since $b \in B_A(a) \cap A$ and $c \in A$ we have
	\begin{align*}
	b+c &\in \left(A+A\right) \cap \left(B_A(a)+c\right)\\
	& \subseteq \left(A+A\right) \cap \left(B_A(a_x)+c_x\right).
	\end{align*}
Since $(a_x,c_x)$ is additively good, there are, as required, at most $2^{j_0+3}|A+A|/|A_{j_0}|$ possibilities for $b+c$. The argument that there are at most at most $2^{j_0+3}|AA|/|A_{j_0}|$ for $bd$ is similar. 

In total therefore
	$$Q \leq \frac{2^{2j_0+6}|A+A|^2|AA|^2}{|A_{j_0}|^2}$$ 
which concludes the proof of (\ref{eq:tom2}) and thus of the lemma.
\end{proof}

The following result shows that any chain contains a large separable subset, allowing Lemma~\ref{theorem:analyseseparable} to be applied to the large chain found in Lemma~\ref{theorem:chains}. 

\begin{lemma}\label{theorem:strict}
If $C$ is the set of elements of an $A$-chain for some $A\subset\ffield$ then $C$ contains a separable set of cardinality at least $|C|/q$.
\end{lemma}

\begin{proof}
It is clear that any subset $\left\{c_1,\ldots,c_n\right\} \subset C$ with
	\[B_A(c_1)\subsetneq \ldots \subsetneq B_A(c_n).\]
is separable. Define an equivalence relation on elements of $A$ by $a \sim b$ if and	 only if $B_A(a)=B_A(b)$. To prove the lemma it suffices to show that each equivalence class contains at most $q$ elements of $A$. 

Note first that if $a \sim b$ and $a\neq b$ then 
	\[\norm{a-b}=r_A(a)=r_A(b).\]
Indeed, since $B_A(a)=B_A(b)$ it follows that $b \in B_A(a)$ and so $\norm{a-b}\leq r_A(a)$. However by minimality, $\norm{a-b}\geq r_A(a)$ and so $\norm{a-b}=r_A(a)$. Similarly $\norm{a-b}=r_A(b)$. 

Suppose for a contradiction that there is an equivalence class containing elements $a_1,\ldots,a_{q+1}$. Consider the differences $a_1-a_i$ for $2 \leq i \leq q+1$. By the last paragraph we have
	\begin{equation}\label{normeq}\norm{a_i-a_j}=r_A(a_i)=r_A(a_j)\end{equation}
for all $1\leq i< j\leq q+1$. To complete the proof it suffices to show that there are distinct $i$, $j$ and $k$ in $\{1,\ldots,q^d+1\}$ such that
\begin{equation}\label{degrees}\deg(a_i-a_j)<\deg(a_i-a_k),\end{equation}
as this yields a contradiction to \eqref{normeq} by the definition of the norm $\norm{\cdot}$. 

Consider the $q$ distinct elements $a_1-a_i$ for $2\leq i\leq q+1$. Since none of these are $0$ they have leading coefficient belonging to $\mathbb{F}_q\backslash\{0\}$, and hence by the pigeonhole principle there are $i\neq j$ such that $a_1-a_i$ and $a_1-a_j$ have the same leading coefficient. 

It follows that if $\deg(a_1-a_i)=\deg(a_1-a_j)$ then $\deg(a_i-a_j)<\deg(a_1-a_j)$. Hence either way there is a choice of distinct $i$, $j$ and $k$ such that \eqref{degrees} holds, completing the proof. 
\end{proof}
Theorem~\ref{theorem:functionfieldsumproduct} follows by combining Lemma~\ref{theorem:analyseseparable} with Lemmata~\ref{theorem:chains} and \ref{theorem:strict} as follows.

\begin{proof}[Proof of Theorem~\ref{theorem:functionfieldsumproduct}]
Let 
	\[K= \frac{|A|^5}{q|A+A|^2|A\cdot A|^2\log^3|A|}.\]
By Lemmata~\ref{theorem:chains} and \ref{theorem:strict} the set $A$ contains a separable subset $U$ of cardinality $\Omega(K)$. For any natural number $k$, Lemma~\ref{theorem:analyseseparable} implies that
	\begin{align*}	
	|kA|&\geq |kU|\gg_{k} K^{k}
	\end{align*}
By Pl\"unnecke's inequality (see \cite{Pe:2011} for a recent simple proof) we know that $|kA|\ll_k \frac{|A+A|^k}{|A|^{k-1}}$ and so, after taking $k$-th roots,
	\begin{align*}
	|A+A|&\gg_k |A|^{1-1/k}K\\
	& = \frac{|A|^{6-1/k}}{q|A+A|^2|A\cdot A|^2\log^3 |A|}
	\end{align*}
Theorem~\ref{theorem:functionfieldsumproduct} follows by taking $k$ sufficiently large.
\end{proof}

\section*{Acknowledgements}
The authors would like to thank Oliver Roche-Newton, Misha Rudnev, and Trevor Wooley for reading earlier drafts of this paper and many helpful comments and corrections.

\end{document}